\documentclass[12pt]{article}
\usepackage{amsmath,amsthm,amsfonts,amssymb,times,hyperref}
\usepackage{xcolor}

\newtheorem{theorem}{Theorem}
\newtheorem{proposition}{Proposition}
\newtheorem{lemma}{Lemma}

\newtheorem*{utheorem}{Theorem}

\theoremstyle{definition}

\allowdisplaybreaks

\title{Small sums of roots of unity}
\author{
Alexandros Kalogirou \\
Erd\H{o}s Center \\
R\'enyi Institute \\
Budapest, H-1053 \\
Email: kalogira@email.sc.edu}

\newcommand{\Q}{\mathbb{Q}}
\newcommand{\Z}{\mathbb{Z}}

\begin{document}
\maketitle
\begin{abstract}
    We address the question of how small a non-vanishing sum of $N$-th roots of unity with $k$ terms can be. We show upper bounds of the shape $N^{-\alpha_k}$, where $\alpha_k\rightarrow\infty$ with $k$. We also address the question of improving these bounds for a positive proportion of $N$. 
\end{abstract}
\section{Introduction}
Let 
\begin{equation}\label{def}
    f(k,N)=\min\left\{|S|: 0\neq S=\sum_{i=1}^k \zeta_i, \text{ with  }\zeta_i^N=1 \text{ for }1\leq i\leq k \right\}.
\end{equation}
In other words, $f(k,N)$ is the modulus of the smallest non-zero sum of $k$ $N$-th roots of unity. The modulus is taken over the complex numbers.

We will be concerned with upper bounds for $f(k,N)$, for all $k$, and primarily under the assumption that $N$ is large and tending to infinity. 

This question can be traced back to Myerson \cite{myerson}, who established that 
\begin{equation}
    k^{-N}\leq f(k,N)\leq N^{-k/4+o_N(1)},
\end{equation}
in the case that both $k$ and $N$ are even. While exact values of $f(k,N)$ are known when $k\in{2,3,4}$ and general $N$, it is for the greater part unknown whether one can achieve a bound of the form $f(k,N)=O(N^{-\alpha})$, with $\alpha>1$, and for all $N$. A notable exception and contribution is that of Barber \cite{benb}, who addressed the case $k=5$, showing in his Theorem 1 that $f(5,N)=O(N^{-4/3})$. The work in \cite{benb} also addresses improvements to the exponent $4/3$ for certain sparse values of $N$ and at the same time provides many useful charts from computations of $f(5,N)$ up to 221000. It also acted as a motivation for the results presented here, giving an excellent summary of the admittedly limited literature on the subject, for instance the methods of Myerson that achieve the bound $f(k,N)=O(N^{-k/2})$ for $N$ even and small even $k$.

Our results have a double goal. One is to show that for general $N$ one gets bounds of the shape $N^{-\alpha_k}$, where for $k\geq 4$, we have $\alpha_k>1$, while at the same time $\alpha_k\rightarrow \infty$ with $k$. The next goal is to show that one can achieve better exponents than $\alpha_k$ for $N$ lying in dense sets, namely multiples of certain numbers that depend on $k$.

The most singular case is that of $k=7$, which we treat separately.

\begin{theorem}
    With the same notation as in \eqref{def}, we have
    \begin{equation}\label{7result}
        f(7,N)=O(N^{-3/2}).
    \end{equation}
\end{theorem}
This will be proved in the second section of this paper.

Next, we have
\begin{theorem}
    We have
    \begin{equation}\label{genk}
        f(k,N)=O(N^{-\lfloor \frac{\log k}{\log 6} \rfloor}).
    \end{equation}
\end{theorem}
The proof of Theorem 2 will be given in section 3. The same method of proof, allows to conclude the bound $f(k,N)=O(N^{-2})$, which is sharper than \eqref{genk} when $11\neq k\leq 36$. When $k=11$, a similar proof as that of Theorem 1 gives $f(11,N)=O(N^{-5/3})$. 

It is worth noting that both proofs rely on perturbing certain configurations of roots of unity, as was introduced by Barber \cite{benb}.

Finally, the last section shows how the method of section 3 can be adapted to improve on the two theorems in special cases.

\section{Proof of Theorem 1}
Our guide will be some classical results on inhomogeneous Diophantine approximation. While these results are already accounted for in the classic book of Cassels \cite{cassels}, we draw convenient references from the survey article of Bugeaud and Laurent \cite{transfer}.

Starting with the notation, let $|\vec{x}|=\max_{i=1}^n\{|x_i| \}$, for a vector $\vec{x}=(x_1,\dots, x_n)\in \mathbb{R}^n$. Also define
\begin{equation*}
    \| \vec{x} \|=\inf_{\vec{y}\in \Z^n}|\vec{x}-\vec{y}|.
\end{equation*}
Given a matrix $A=(a_{i,j})$ with $n$ rows and $m$ columns, consider the following linear forms defined by the columns of $A$:
\begin{equation}
    M_j(\vec{y})=\sum_{i=1}^n a_{i,j}y_i,\,\,\, \vec{y}=(y_1,\dots,y_n)   ,\,\,\, (1\leq j \leq m).
\end{equation}
We define
\begin{equation}
    M(\vec{y})=\max_{1\leq j \leq m}\| M_j(\vec{y}) \|=\|A^T\vec{y}\|.
\end{equation}
Finally, let 
\begin{equation}
    L_i(\vec{y})=\sum_{j=1}^m a_{i,j}x_j,\,\,\, \vec{x}=(x_1,\dots,x_m)   ,\,\,\, (1\leq i \leq n),
\end{equation}
be the linear forms determined by the rows of the matrix $A$. 

The following is almost verbatim Lemma 3 of \cite{transfer}.
\begin{lemma}
    Let $\kappa=2^{1-m-n}((m+n)!)^2$. Let $X$ and $Y$ be two positive real numbers. Suppose that we have the lower bound
    \begin{equation}
        M(\vec{y})\geq \kappa X^{-1},
    \end{equation}
    for any non-zero $\vec{y}\in\Z^n$ with $|\vec{y}|\leq Y$. Then for all real vectors $\vec{\theta}=(\theta_1,\dots,\theta_n)\in\mathbb{R}^n$, there exists $\vec{x}\in\Z^m$ with $|\vec{x}|\leq X,$ such that 
    \begin{equation}\label{inhom}
        \max_{1\leq i\leq n} \|L_i(\vec{x})-\theta_i   \|\leq \kappa Y^{-1}.
    \end{equation}
\end{lemma}
It is very crucial to us, that the conclusion of the lemma holds uniformly for all vectors $\vec{\theta}$.

In what follows, we will be using the following convenient abbreviations.
\begin{align*}
    e(x)&=e^{2\pi i x},\\
    c(x)&=\cos(2\pi x),\\
    s(x)&=\sin(2\pi x).
\end{align*}

The above lemma is the key ingredient in proving our next Proposition 1. Before doing that, we record here a lemma from Calculus.
\begin{lemma}

Suppose that $x,y,z$ satisfy
\[
1+2c(x)+2c(y)+2c(z)=0,
\]
where
\[
x\in\left[\frac13,\frac13+\delta\right],\qquad
y\in\left[\frac16,\frac16+\delta\right],
\]
with $\delta>0$ small, and $z$ is chosen in $[0,\frac12]$. Then, the map
\[
F(x,y)=\left(\frac{s(x)}{s(z)},\,\frac{s(y)}{s(z)}\right)
\]
has image containing an open $\mathbb{R}^2$ ball of $(1,1)$.
\end{lemma}

\begin{proof}
Let
\[
G(x,y,z)=1+2c(x)+2c(y)+2c(z).
\]
At the point $(x_0,y_0,z_0)=\left(\frac13,\frac16,\frac13\right)$,

we have
\[
G(x_0,y_0,z_0)=0,
\]
and
$\frac{\partial G}{\partial z}=-4\pi s(z),$ which gives $\frac{\partial G}{\partial z}(x_0,y_0,z_0)
=-2\pi\sqrt3\neq0.$ Hence, by the Implicit Function Theorem, there is a neighborhood of
$\left(\frac13,\frac16\right)$ on which $z=z(x,y)$ is a smooth function
satisfying
\begin{equation}\label{implicit}
G(x,y,z(x,y))=0.
\end{equation}

We therefore regard
\[
F(x,y)
=
\left(
\frac{s(x)}{s(z(x,y))},
\frac{s(y)}{s(z(x,y))}
\right)
\]
as a smooth map of the two independent variables $x,y$. After solving \eqref{implicit} for $z_x$ and $z_y$, we find the Jacobian of the map $F(x,y)$ at the point $(1/3,1/6)$ to be

\[
JF\!\left(\frac13,\frac16\right)
=
\begin{pmatrix}
-\dfrac{4\pi}{\sqrt3} & -\dfrac{2\pi}{\sqrt3}\\[1.2ex]
-\dfrac{2\pi}{\sqrt3} & 0
\end{pmatrix},
\]
whose determinant is
\[
\det JF\!\left(\frac13,\frac16\right)
=
-\frac{4\pi^2}{3}\neq0.
\]

 Using the Inverse Function Theorem again,
$F$ is a local diffeomorphism at this point. Therefore the image of $F$ contains an open neighborhood of $F(1/3,1/6)=(1,1)$
\end{proof}
\begin{proposition}
    There exist $\alpha_1,\alpha_2,\alpha_3\in \mathbb{R}$, such that 
    \begin{equation}\label{zero-sum}
        1+\sum_{i=1}^3 \left[e(\alpha_i)+e(-\alpha_i)\right]=0,
    \end{equation}
    which satisfy the following. Let $L(x_1,x_2)=x_1s(\alpha_1)/s(\alpha_3)+x_2s(\alpha_2)/s(\alpha_3)$. Then for any $\theta\in \mathbb{R}$, one can find a vector $\vec{x}$ with $|\vec{x}|\leq X$, such that
    \begin{equation}
        \|L(\vec{x})-\theta \|\ll_{\vec{\alpha}} X^{-2}.
    \end{equation}
\end{proposition}
\begin{proof}
     By Lemma 1, it is enough to find a 2-tuple $A=(s(\alpha_1)/s(\alpha_3),s(\alpha_2)/s(\alpha_3))$, such that
    \begin{equation}\label{sufficient}
        \max_{1\leq i\leq 2}\{\| s(\alpha_i)/s(\alpha_3)y \|\}\geq c Y^{-1/2},
    \end{equation}
    whenever $y$ is an integer with $|y|\leq Y$. The seminal work of Schmidt \cite{schmidt} has as consequence that points $(t_1,t_2)$ satisfying
    $$ \max_{1\leq i\leq  2}\{\| t_iy\|\}\geq c_{t_1,t_2} Y^{-1/2}  $$
    for $y\in \Z$ with $|y|\leq Y$, exist withing any open ball of $\mathbb{R}^2$. By Lemma 2, we see then that such a point $(t_1,t_2)$ can be made to coincide with $(s(\alpha_1)/s(\alpha_3),s(\alpha_2)/s(\alpha_3))$ for a vector $(\alpha_1,\alpha_2,\alpha_3)$ that satisfies \eqref{zero-sum}.
\end{proof}
Before we present the proof of Theorem 1, a few comments are in order. We will use the same strategy as in \cite{benb}, i.e. finding small non-zero sums by perturbing sums which are already vanishing. 

A potential pitfall in this strategy is that we need to ensure that the new sum is non-vanishing. A more ad hoc reasoning for how this can be avoided can be attempted for the case of small sums,
$$\sum_{i=1}^k e(x_i), \text{  with  } x_i\in\Q  $$
as say, Barber did with the case $k=5$. When perturbing rational solutions to the above equation, the following result of Conway and Jones provides an answer.
\begin{utheorem}
    The equation
    $$  \sum_{i=1}^k x_i=0 $$
    has only finitely many solutions up to rotations, with the $x_i$'s being roots of unity, if we assume that none of the subsums are vanishing.
\end{utheorem}
We see then that we can avoid the vanishing of the new sum if the perturbations are such that for big enough $N,$ the resulting vectors form angles that fall into neighborhoods of the angles of the original vanishing sum, whose diameters tend to zero, while being different from their original values. That is because by the above theorem, the angles of different configurations of vanishing sums $\sum_{i=1}^k e(x_i)=0,$ are separated.

Since we are forced to also consider real solutions to the equation $\sum_{i=1}^7 e(x_i)=0$ to prove Theorem 1, we need to split the argument into two cases.

\textit{Case 1.} The equation $\sum_{i=1}^7 e(x_i/N)$ has a solution $(x_1,\dots,x_7)\in \Z^7$.

\textit{Case 2.} The same equation does not have any integer solutions.

\textit{Case 1} is covered by the following proposition.
\begin{proposition}
    Suppose that there exist integer solutions $x_1,\dots,x_k$ to $\sum_{i=1}^ke(x_i/N)=0$. Then
    $$ f(k,N)\ll 1/N^{\frac{2k-2}{k+1}}  .$$
\end{proposition}
\begin{proof}
    Let $S_0=\sum_{i=1}^k e(x_i/N)=0 $. Consider the perturbation by $\vec{y}=\frac{1}{N}(y_1,\dots,y_k)$, with $\max_{1\leq i\leq k}\{|y_i|\}<Q$, which gives a new sum
    \begin{align}
        S&=\sum_{i=1}^k e(x_i/N+y_i/N)\notag\\
        &=\sum_{i=1}^k c(x_i/N+y_i/N)+i\left(\sum_{i=1}^k s(x_i/k+y_i/N)  \right)\notag\\
        &=\frac{-2\pi}{N}\left( \sum_{i=1}^k s(x_i/k)*y_i \right)+\frac{2\pi i}{N}\left( \sum_{i=1}^kc(x_i/k)*y_i  \right)+O\left( \frac{Q^2}{N^2} \right).\label{taylor}
    \end{align}
    The line \eqref{taylor} follows from the Taylor expansion of the functions $c(x),s(x)$ at the points $x_i/N$. Consider the two vectors $\vec{p}=\sum_{i=1}^k s(x_i/k)$ and $\vec{q}=\sum_{i=1}^k c(x_i/k)$. We want to make both inner products $\vec{p}*\vec{y}, \vec{q}*\vec{y}$, small, for some $\vec{0}\neq\vec{y}$, that is not parallel to $\Lambda=(1,1,\dots,1)$. The last requirement is placed to avoid $S$ being zero, and from our previous analysis, that is the only condition that we need to impose on $\vec{y}.$

    Take a total of $\gg Q^{k-1}$ cosets $\vec{w}+\Lambda$ of $\Z/\langle \Lambda \rangle$, whose coordinates $|y_i|$ are in absolute value at most $Q$. The values of $(\vec{p}*\vec{w},\vec{q}*\vec{w})$ are distributed in a $Q\times Q$ box. Therefore by pigeonhole, we can find two vectors $\vec{w},\vec{w'},$ such that both inner products at each vector differ by at most $c\cdot Q/Q^{\frac{k-1}{2}}.$ Their difference $\vec{y}=\vec{w}-\vec{w'}$ will be a non-diagonal vector such that
    $$  \max\{ |\vec{p}*\vec{y}|, |\vec{q}*\vec{y}|  \}\ll Q^{\frac{3-k}{2}}.  $$
    Equating 
    $$  1/(Q^{\frac{k-3}{2}}N)\approx Q^2/N^2,  $$
    we find that $S\ll 1/N^{\frac{2k-2}{k+1}}. $
\end{proof}

\begin{proof}[\textit{Proof of Theorem 1.}]
    If we are in \textit{Case 1}, the conclusion follows from Proposition 2. Therefore we assume that we are in \textit{Case 2}. We perturb the configuration of the zero sum in \eqref{zero-sum}, as follows. Let $x_i/N$ for $x_i\in \Z$, be the closest point to $\alpha_i$, among all points of the lattice $\frac{1}{N}\Z$. We then consider the points $y_i/N$, with $|y_i-x_i|\leq Q$ for $i=1,2$, while allowing the potentially bigger range $|y_3-x_3|\leq \sum_{i=1}^2 |s(\alpha_i)/s(\alpha_3)|(Q+1) $ for $i=3$. Here $Q$ is some parameter to be chosen later.

    We now consider the sum
    \begin{align}
        S&=1+\sum_{i=1}^3[e(y_i/N)+e(-y_i/N)]\\
        &=1+2\sum_{i=1}^3 c(y_i/N) \notag \\
        &=1+2\sum_{i=1}^3 c(y_i/N-x_i/N+x_i/N-\alpha_i+\alpha_i)\notag\\
        &=1+2\sum_{i=1}^3 c(\alpha_i)-4\pi\sum_{i=1}^3s(\alpha_i)\left( \frac{y_i}{N}-\frac{x_i}{N}+\frac{x_i}{N}-\alpha_i \right)+O\left(\frac{Q^2}{N^2}\right) \label{Taylor} \\
        &=-\frac{4\pi s(\alpha_3)}{N}\left[\sum_{i=1}^2\frac{s(\alpha_i)}{s(\alpha_3)}\left( y_i-x_i)  \right)+\sum_{i=1}^3\frac{s(\alpha_i)}{s(\alpha_3)}(x_i-N\alpha_i)+y_3-x_3  \right]+O\left(\frac{Q^2}{N^2}  \right)\label{errors}
    \end{align}
    The line \eqref{Taylor} comes from Taylor expanding the function $c(x)$ around the points $\alpha_i$. Setting $\theta=-\sum_{i=1}^3s(\alpha_i)/s(\alpha_3)(x_i-N\alpha_i)$, we first observe that $\theta$ is absolutely bounded by $\sum_{i=1}^2|s(\alpha_i)/s(\alpha_3)|$, regardless of the value of $N$. We therefore can bound the bracket above by
    \begin{equation}
        c\| L(y_1-x_1,y_2-x_2) -\theta \|,
    \end{equation}
    for suitable choice of $y_3-x_3$. By Proposition 1, this can be seen to be $O(1/Q^2)$. Equating the two error terms in \eqref{errors}, we find the choice of $Q\sim N^{1/2}$ to be optimal, which gives
    $$  S=O(N^{-3/2}). $$
\end{proof}
\section{Proof of Theorem 2}
Consider the sums
\begin{equation}
    (1+\zeta)^a,
\end{equation}
and
\begin{equation}\label{sum1}
    (1+\omega+\omega^2)^b,
\end{equation}
each with $2^a$ terms and $3^b$ terms accordingly. By an easy application of the Frobenius coin problem, as long as
\begin{equation}\label{sum2}
    2^a3^b\leq k,
\end{equation} 
then $k$ can be written as $k=2^ap+3^bq$ for non-negative integers $p,q$. Pick an arbitrary non-vanishing sum $S_1$ of $N$-th roots of unity with $p$ terms, and another $S_2$ with $q$ terms. Taking $\zeta$ to be an $N$-root of unity close to -1, but not equal, and $\omega$ similarly to be an $N$-th root close to, but different from $e(1/3)$, the sum
$$ (1+\zeta)^a*S_1+(1+\omega+\omega^2)^b*S_2,  $$
is of order $N^{-\min\{a,b\}}$. We can also make sure that the sum is non-zero, by rotating if necessary any of the sums $S_1$ or $S_2$ by multiplication by any $N$-root different from 1. The choice $a=b$ gives
$$ f(k,N)\ll N^{-\lfloor \frac{\log k}{\log 6} \rfloor}.  $$
\section{Improvements conditional on divisors of $N$}
We illustrate the more general idea by applying it to specific cases of interest.
\begin{proposition}
    Assume $6|N$. Then $f(5,N)\ll N^{-2}$.
\end{proposition}
\begin{proof}
    Consider the sum
    \begin{align*}(1+\zeta)(1+\zeta'\omega+\zeta'\omega^2)=1+\zeta'\omega&+\zeta'\omega^2 +\\&+\zeta+\zeta\zeta'\omega+\zeta\zeta'\omega^2. \end{align*}
    Taking $\omega=e(1/3)$, $\zeta=\zeta'=(1/2+1/N)$, we have as before a sum of order $N^{-2}$, while we can write $\zeta'\omega+\zeta'\omega^2=-\zeta'$,in other words, a sum with five terms, all of which are $N$-th roots of unity.
\end{proof}
\begin{proposition}
    Assume $10\mid N$. Then $f(7,N)\ll N^{-2}$.
\end{proposition}
\begin{proof}
    Like before, consider the sum
    \begin{align*}
        (1+\zeta)(1+\zeta'*\sum_{i=1}^4 \omega^i )=1+\zeta'*\sum_{i=1}^4\omega^i+\zeta+\zeta*\zeta'*\sum_{i=1}^4\omega^i.
    \end{align*}
    Taking $\omega=e(1/5)$, $\zeta'=\zeta=e(1/2+1/N)$, after replacing $\zeta'*\sum_{i=1}^4\omega^i=-\zeta'$, we get the conclusion.
\end{proof}
Finally, the following proposition shows a little more originality in the use of identities to exploit the cancellation coming from longer sums. Remarkably, this matches the best result of Myesron for $k=8$, and $N$ even, while using a construction that involves solving linear equations as opposed to his approach involving solutions to the Prouhet-Tarry-Escott problem.
\begin{proposition}
    Assume $2|N$. Then $f(10,N)\ll N^{-4}$.
\end{proposition}
\begin{proof}
    Consider the set $A=\{ -3,-2,-1,1 \}$, and the sum given by the expansion of
    \begin{align*}
        \prod_{a\in A}(1-e(a/N))=2-e(1/N)-e(-2/N)-e(-3/N)+2e(-5/N)-e(-6/N)..
    \end{align*}
    This is a sum of eight roots of unity, and of order $O(N^{-4})$. The set $A$ was chosen so that $0\notin A$, and $|A \cap (A+'A) |=3$. Here $A+'A$ stands for the multiset $A+'A=\{a+a'|a, a'\in A, a\neq a' \}$. 
\end{proof}

\section{Acknowledgements}
The author wants to thank Christian Bernert for suggesting the problem, together with the reference \cite{benb}.

\end{document}